\documentclass[11pt,a4paper,reqno]{amsart}
\usepackage[margin=2cm]{geometry}
\usepackage{hyperref,amssymb,amsfonts,amsthm,amsrefs,mathtools,verbatim,microtype,bm}
\usepackage{stmaryrd}
\usepackage[shortlabels]{enumitem}
\usepackage{cleveref}

\newtheorem{Theorem}{Theorem}[section]
\crefname{Theorem}{Theorem}{Theorems}

\newtheorem{Proposition}[Theorem]{Proposition}
\crefname{Proposition}{Proposition}{Propositions}

\newtheorem{Lemma}[Theorem]{Lemma}
\crefname{Lemma}{Lemma}{Lemmas}

\newtheorem{Corollary}[Theorem]{Corollary}
\crefname{Corollary}{Corollary}{Corollaries}

\newtheorem*{Claim*}{Claim}

\theoremstyle{definition}
\newtheorem{Definition}[Theorem]{Definition}
\crefname{Definition}{Definition}{Definitions}

\crefname{Example}{Example}{Examples}

\crefname{Fact}{Fact}{Facts}

\crefname{Remark}{Remark}{Remarks}

\crefname{Conjecture}{Conjecture}{Conjectures}

\crefname{Question}{Question}{Questions}

\newtheorem*{Questions*}{Questions}
\newtheorem*{FQuestions*}{Further Questions}

\makeatletter
\newtheorem*{rep@theorem}{\rep@title}
\newcommand{\newreptheorem}[2]{%
\newenvironment{rep#1}[1]{%
 \def\rep@title{#2 \ref{##1}}%
 \begin{rep@theorem}}%
 {\end{rep@theorem}}}
\makeatother

\newreptheorem{theorem}{Theorem}

\DeclareFontFamily{U}{mathb}{\hyphenchar\font45}
\DeclareFontShape{U}{mathb}{m}{n}{
<-6> mathb5 <6-7> mathb6 <7-8> mathb7
<8-9> mathb8 <9-10> mathb9
<10-12> mathb10 <12-> mathb12
}{}
\DeclareSymbolFont{mathb}{U}{mathb}{m}{n}
\DeclareMathSymbol{\pprec}{\mathrel}{mathb}{"CE}
\DeclareMathSymbol{\ssucc}{\mathrel}{mathb}{"CF}

\DeclareFontFamily{U}{mathb}{\hyphenchar\font45}
\DeclareFontShape{U}{mathb}{m}{n}{
      <5> <6> <7> <8> <9> <10> gen * mathb
      <10.95> mathb10 <12> <14.4> <17.28> <20.74> <24.88> mathb12
      }{}
\DeclareSymbolFont{mathb}{U}{mathb}{m}{n}
\DeclareFontSubstitution{U}{mathb}{m}{n}
\DeclareMathSymbol{\monus}{2}{mathb}{"01}

\DeclareMathOperator{\dom}{\mathrm{dom}}

\DeclareMathOperator{\degT}{\mathrm{deg}_{\mathrm{T}}}
\DeclareMathOperator{\dgsp}{\mathrm{DgSp}}

\newcommand{\pa}{\mathrm{PA}}
\newcommand{\dnc}{\mathrm{DNC}}
\newcommand{\dnr}{\mathrm{DNR}}

\newcommand{\la}{\langle}
\newcommand{\ra}{\rangle}
\newcommand{\da}{{\downarrow}}
\newcommand{\ua}{{\uparrow}}
\newcommand{\imp}{\rightarrow}
\newcommand{\Imp}{\Rightarrow}

\newcommand{\Biimp}{\Leftrightarrow}

\newcommand{\iso}{\cong}

\newcommand{\Nb}{\mathbb{N}}
\newcommand{\Zb}{\mathbb{Z}}
\newcommand{\Qb}{\mathbb{Q}}

\newcommand{\smf}{\smallfrown}
\newcommand{\rst}{{\restriction}}
\newcommand{\keq}{\simeq}
\newcommand{\forae}{\forall^\infty}

\newcommand{\mc}[1]{\mathcal{#1}}
\newcommand{\mf}[1]{\mathfrak{#1}}

\newcommand{\geqT}{\geq_\mathrm{T}}

\newcommand{\std}{\mathrm{std}}
\newcommand{\nonstd}{\mathrm{nonstd}}

\AtBeginDocument{%
   \def\MR#1{}
}

\title[Tennenbaum-like theorems for cohesive powers]{Tennenbaum-like theorems for cohesive powers}

\author{David Gonzalez}
\address{Department of Mathematics\\
University of Notre Dame\\
Hurley Hall, 255 Hurley\\
Notre Dame, IN 46556\\
United States of America}
\email{dgonza42@nd.edu}
\urladdr{https://www.davidgonzalezlogic.com}

\author{Paul Shafer}
\address{School of Mathematics\\
University of Leeds\\
Leeds\\
LS2 9JT\\
United Kingdom}
\email{p.e.shafer@leeds.ac.uk}
\urladdr{https://peshafer.github.io}

\date{\today}

\begin{document}

\begin{abstract}
We investigate the encoding ability of the cohesive power construction.  We compute a graph $\mc{G}$ where the cohesive power $\prod_C \mc{G}$ of $\mc{G}$ by any $\Delta_2$ cohesive set $C$ has degree $0''$.  That is, $0''$ computes a presentation of $\prod_C \mc{G}$, and every presentation of $\prod_C \mc{G}$ computes $0''$.  We also compute a linear order $\mc{L}$ where no cohesive power of $\mc{L}$ has a computable presentation.  We accomplish this by ensuring that if $\mc{P}$ is a presentation of a cohesive power of $\mc{L}$, then $\mc{P}''$ has $\pa$-degree relative to $0''$. 
\end{abstract}

\maketitle

\section{Introduction}

We investigate the encoding ability of the cohesive power construction.  A \emph{cohesive power} is a miniature power of a computable structure $\mc{A}$ built from a cohesive subset of $\Nb$ and partial computable functions $\Nb \to |\mc{A}|$ instead of from an ultrafilter on $\Nb$ and all functions $\Nb \to |\mc{A}|$.  These powers are countable structures because there are only countably many partial computable functions. 
For this reason, they can be analyzed by traditional techniques from computable structure theory that focus on countable structures, as we do in this article.

Close relatives of cohesive powers of the standard model of arithmetic $\mc{N} = (\Nb; 0, 1, +, \times, <)$ have been widely studied for a long time.  Even Skolem's original construction of a countable non-standard model of arithmetic~\cite{Skolem} may be seen as a miniature ultrapower.  Skolem's model, and indeed all countable non-standard models of arithmetic, were shown to have no computable copies by Tennenbaum \cite{Ten51}.  Even more closely related to cohesive powers are the \emph{recursive ultrapowers} (aka \emph{Nerode semi-rings}), \emph{r.e.\ ultrapowers}, and similar constructions studied in works including~\cites{FefermanScottTennenbaum, HirschfeldModels, HirschfeldWheelerBook, ShavrukovPrimePowers, McLaughlinRearrangement, McLaughlinCollectionFails, McLaughlinSurvey, McLaughlinUltrapowersModelsExtensions, McLaughlinTotallyRigid, McLaughlinEmbeddings, LermanCo-r-Max}.  
Shavrukov~\cite{ShavrukovPrimePowers} observes that cohesive powers as studied here coincide with his \emph{r.e.\ prime powers}.
The idea to study powers of arbitrary computable structures via cohesive sets is more recent, originating with Dimitrov~\cite{DimitrovCohPow}.

{\L}o\'{s}'s theorem for cohesive powers (see \cref{thm-LosGeneral} below) only holds for sentences that are $\Delta_3$-expressible, and a cohesive power of $\mc{N}$ will not be a model of Peano arithmetic.  In fact, $\mathsf{B}\Sigma_1$ fails in cohesive powers of $\mc{N}$~\cite{ShavrukovPrimePowers}*{Proposition~2.11}.  Nevertheless, preserving $\Delta_3$ facts is enough for Tennenbaum's no-computable-copy theorem to apply.  To wit, Shavrukov~\cite{ShavrukovPrimePowers} shows that cohesive powers of $\mc{N}$ model the $\Pi_2$ theory of true arithmetic, and it follows that Tennenbaum's theorem applies.  We give the details using {\L}o\'{s}'s theorem for cohesive powers, which we explain more fully later on.

\begin{Theorem}[essentially~\cite{ShavrukovPrimePowers}*{Corollary~2.8}]
Let $\mc{N} = (\Nb; 0, 1, +, \times, <)$ be the standard model of arithmetic, and let $C \subseteq \Nb$ be a cohesive set.  Then the cohesive power $\prod_C \mc{N}$ of $\mc{N}$ over $C$ does not have a computable presentation.
\end{Theorem}

\begin{proof}
\emph{A priori}, when discussing arithmetic in this context and intending to apply a version of {\L}o\'{s}'s theorem for cohesive powers, one must take care to distinguish between the formula classes $\Sigma_n$ and $\Pi_n$, where the matrix formula is allowed bounded quantifiers, and the formula classes $\exists_n$ and $\forall_n$, where the matrix formula is quantifier-free.  However, $\Delta_0$ formulas in arithmetic are decidable in $\mc{N}$, and therefore \cref{lem-LosProdParamHelper} applies to them.  Specifically, \cref{lem-LosProdParamHelper} item~\ref{it-LosProdPramPi2Helper} implies that $\prod_C \mc{N}$ satisfies the $\Pi_2$ theory of true arithmetic.  It follows that $\prod_C \mc{N} \models \mathsf{I}\Delta_0$.  Tennenbaum's theorem applies to non-standard models of $\mathsf{I}\Delta_0$ (see, for example, \cite{KayeBook}*{Theorem~11.8}), and $\prod_C \mc{N}$ is non-standard because the identity function represents a non-standard element.  Therefore $\prod_C \mc{N}$ does not have a computable presentation.
\end{proof}

The present article is motivated by Tennenbaum's theorem and the investigations in~\cites{CohPowJournal, ShaferCohPowDelta2} concerning cohesive powers of computable copies of the linear order $(\Nb; <)$.  These works examine how the theory of cohesive powers deviates from the classical theory of ultrapowers, even for basic structures like $(\Nb; <)$.  For the usual presentation of $(\Nb; <)$, where the successor function is computable, and any cohesive set $C$, the cohesive power $\prod_C (\Nb; <)$ has order-type $\omega + \zeta\cdot\eta$.
 Here $\omega$, $\zeta$, and $\eta$ denote the order-types of $(\Nb; <)$, $(\Zb; <)$, and $(\Qb; <)$.
This is expected, as $\omega + \zeta\cdot\eta$ is the order type of any non-standard model of arithmetic and is elementary equivalent to $\omega$.  However, it is possible to compute a linear order $\mc{L}$ that is classically (but not computably) isomorphic to $(\Nb; <)$ where for some cohesive sets $C$, the cohesive power $\prod_C \mc{L}$ does not have order-type or even the same elementary class as $\omega + \zeta\cdot\eta$.  For example, there is a fixed computable copy $\mc{L}$ of $\omega$ such that $\prod_C \mc{L} \iso \omega + \eta$ for every $\Delta_2$ cohesive set $C$.  While this is perhaps surprising, it is not Tennenbaum-like; even though $\omega + \eta$ is quite different from $\omega$, it still has a computable copy.  More than just $\omega + \eta$ can be produced in this way, however. The most general theorem known in this direction is the following.

\begin{Theorem}[\cite{ShaferCohPowDelta2}*{Theorem~4.5}]\label{thm-BCSigma2ShuffleNew}
Let $X \subseteq \Nb \setminus \{0\}$ be a Boolean combination of $\Sigma_2$ sets, thought of as a set of finite order-types.  Then there is a computable copy $\mc{L}$ of $\omega$ such that for every $\Delta_2$ cohesive set $C$, the cohesive power $\prod_C \mc{L}$ has order-type $\omega + \bm{\sigma}(X \cup \{\omega + \zeta \cdot \eta + \omega^*\})$.  Moreover, if $X$ is finite and non-empty, then there is also a computable copy $\mc{L}$ of $\omega$ such that for every $\Delta_2$ cohesive set $C$, the cohesive power $\prod_C \mc{L}$ has order-type $\omega + \bm{\sigma}(X)$.
\end{Theorem}

\noindent
Here, $\bm{\sigma}$ denotes the \emph{shuffle sum}, and $\omega^*$ denotes the reverse of $\omega$.  The shuffle sum $\bm{\sigma}(\{\mc{L}_i\}_{i \in \Nb})$ of a countable collection of linear orders is obtained by coloring $(\Qb; <)$ with countably many colors in such a way that every color is dense and then replacing each point of color $i$ by a copy of order $\mc{L}_i$.

It is natural to wonder if any of these constructions are Tennenbaum-like, in other words, if any of these linear orders fail to have computable copies.  A \emph{block} in a linear order $\mc{L}$ is a maximal set $B \subseteq |\mc{L}|$ such that whenever $a$ and $b$ are in $B$ with $a \prec_\mc{L} b$, the interval $[a, b]$ is finite.  For a computable linear order $\mc{L}$, the set $\{n : \text{$\mc{L}$ has a block of size $n$}\}$ is $\Sigma_3$.  If we identify $X \subseteq \Nb \setminus \{0\}$ with the corresponding set of finite order-types, then $X$ is exactly the sizes of the finite blocks of $\omega + \bm{\sigma}(X \cup \{\omega + \zeta\cdot\eta + \omega^*\})$.  If $X$ is a Boolean combination of $\Sigma_2$ sets as in \cref{thm-BCSigma2ShuffleNew}, or indeed even a $\Sigma_3$ set, then $\bm{\sigma}(X \cup \{\omega + \zeta\cdot\eta + \omega^*\})$ and hence $\omega + \bm{\sigma}(X \cup \{\omega + \zeta\cdot\eta + \omega^*\})$ has a computable copy by (a slight modification of) \cite{AshJockuschKnight}*{Lemma~2.1}.  Indeed, all of the order-types produced by \cref{thm-BCSigma2ShuffleNew} have computable copies, so are not Tennenbaum-like.  On the other hand, if $X$ is $\Pi_3$ but not $\Sigma_3$, then $\omega + \bm{\sigma}(X \cup \{\omega + \zeta\cdot\eta + \omega^*\})$ is not a computable order-type.  This situation raises several questions.

\begin{Questions*}\
\begin{enumerate}[(1)]
\item\label{q-Pi3} In \cref{thm-BCSigma2ShuffleNew}, can `a Boolean combination of $\Sigma_2$ sets' be improved to `a $\Pi_3$ set?'

\medskip

\item\label{q-omega} Is there a computable copy $\mc{L}$ of $\omega$ where the cohesive power $\prod_C \mc{L}$ has no computable presentation for some cohesive set $C$?  For all cohesive sets $C$?

\medskip

\item\label{q-lo} Is there a computable linear order $\mc{L}$ (not necessarily of order-type $\omega$) where the cohesive power $\prod_C \mc{L}$ has no computable presentation for some cohesive set $C$?  For all cohesive sets $C$?

\medskip

\item\label{q-enc} More generally, what is the encoding strength of the cohesive power construction?
\end{enumerate}
\end{Questions*}

We answer Question~\ref{q-lo} and address the open-ended Question~\ref{q-enc}.  Questions~\ref{q-Pi3}~and~\ref{q-omega} remain open.

These questions are of particular interest because of the special place that linear orders have in computable structure theory.  Linear orders are a rich class of structures where a lot of complicated constructions are possible.  More formally, linear orders are on top for Borel reducibility, so any other structure can be transformed into a linear order in a concrete way that preserves isomorphism \cite{FS89}.  In contrast, unlike ultra-rich structures like graphs or groups, linear orders are unable to computably interpret every other type of structure \cites{Gao01, GHT}.  This means that any transformation of a structure into a linear order necessarily destroys some computability-theoretic information about the structure in the process.  Linear orders are complicated enough to be interesting but not so complicated that you can get away with an argument that focuses on coding a different type of structure.  They merit individual consideration.  Linear orders are an ideal place to test the understanding of new ideas and push the sophistication of the known constructions.

The study of what degrees can compute a structure, or the degree spectrum of a structure, is quite important in computable structure theory (see~\cite{MontalbanBookWithin}*{Chapter~V}).  Unsurprisingly, the degree spectra of linear orders have also been the specific subject of study in works such as~\cites{Ric81, Kni86, JockuschSoareLO, AshJockuschKnight, Mil01}.  By addressing the questions regarding cohesive powers of linear orders, this article falls into this tradition.  In particular, we demonstrate a new way to construct linear orders that do not have computable copies and therefore have interesting degree spectra.

This article is organized as follows.  \Cref{sec-bak} presents some background material on computable structures and cohesive powers.  \Cref{sec-graph} addresses Question~\ref{q-enc}.  We show that there is a computable structure that can smuggle arbitrary information into its cohesive powers via the cohesive sets (\cref{prop-EncodeAll}).  However, if one restricts to $\Delta_n$ cohesive sets (for $n \geq 2$), then the cohesive power of a computable structure over a $\Delta_n$ cohesive set always has a $\Delta_{n+1}$ presentation (\cref{prop-Delta3Pres}).  Conversely, there is a computable graph $\mc{G}$ such that for every cohesive set $C$, every presentation of $\prod_C \mc{G}$ computes $0''$ (\cref{thm-MaxGraphPow}).  For this graph $\mc{G}$ and any $\Delta_2$ cohesive set $C$, the cohesive power $\prod_C \mc{G}$ therefore has degree $0''$:  $0''$ computes a presentation of $\prod_C \mc{G}$, and every presentation of $\prod_C \mc{G}$ computes $0''$.  \Cref{sec-lo} answers Question~\ref{q-lo}.  We compute a linear order $\mc{L}$ such that $\prod_C \mc{L}$ has no computable presentation for any cohesive set $C$.  In fact, $\mc{L}$ has the property that for every cohesive set $C$, the double-jump $\mc{P}''$ of every presentation $\mc{P}$ of $\prod_C \mc{L}$ has $\pa$-degree relative to $0''$ (\cref{thm-LONonRec}).  Finally, \Cref{sec-q} raises a few further questions.

\section{Background}\label{sec-bak}

We refer the reader to~\cites{LermanBook,SoareBookRE} for computability theory in general and to~\cites{AshKnightBook,MontalbanBookWithin,MontalbanBookBeyond} for computable structure theory in particular.  We also refer the reader to~\cite{RosBook} for the theory of linear orders.

\subsection*{Computability notation and definitions}

Partial computable (aka partial recursive) functions $\Nb \to \Nb$ are denoted by $\varphi$, $\psi$, etc.  For a partial computable $\varphi : \Nb \to \Nb$, $\varphi(n)\da$ means that $\varphi$ halts on input $n$, and $\varphi(n)\ua$ means that $\varphi$ does not halt on input $n$.  Let $\dom(\varphi) = \{n : \varphi(n)\da\}$.  A set $W \subseteq \Nb$ is \emph{computably enumerable} (aka \emph{recursively enumerable}) if $W = \dom(\varphi)$ for some partial computable $\varphi$.  The c.e.\ sets are also exactly the sets that are $\Sigma_1$-definable in arithmetic.  Let $(\varphi_e)_{e \in \Nb}$ denote the usual effective enumeration of all partial computable functions $\Nb \to \Nb$.  Sometimes we also write $\varphi_0, \dots, \varphi_{m-1}$ to denote a sequence of $m$ partial computable functions; the usage will be clear from context.  For each $n \geq 2$, $\la x_0, \dots, x_{n-1} \ra \colon \Nb^n \imp \Nb$ denotes the usual computable bijective $n$\nobreakdash-tupling function, and for each $i < n$, $\pi_i$ denotes the projection function onto coordinate $i$.  For $\sigma, \tau \in \Nb^{<\Nb}$ and $f \in \Nb^\Nb$, $\sigma \sqsubseteq \tau$ denotes that $\sigma$ is an initial segment of $\tau$, and $\sigma \sqsubseteq f$ denotes that $\sigma$ is an initial segment of $f$.  For $f \in \Nb^\Nb$ and $n \in \Nb$, $f \rst n = (f(0), f(1), \dots, f(n-1))$ denotes the initial segment of $f$ of length $n$.  We identify $\mc{P}(\Nb)$ with $2^\Nb$ and often view their elements as infinite binary sequences.

For sets $A$ and $B$, $A \subseteq^* B$ means that $A \setminus B$ is finite:  $A$ is a subset of $B$ except for finitely many exceptions.  A set $C \subseteq \Nb$ is \emph{cohesive} if for every c.e.\ set $W$, either $C \subseteq^* W$ or $C \subseteq^* \Nb \setminus W$.  A cohesive set cannot be $\Sigma_1$, but a famous theorem of Friedberg (see~\cite{SoareBookRE}*{Theorem~X.3.3}) guarantees that there are $\Pi_1$ cohesive sets.

\Cref{sec-lo} makes use of $\pa$ degrees and separating sets.  A function $f \colon \Nb \to 2$ is $\dnc_2$ (aka $\dnr_2$) relative to a function $g \colon \Nb \to \Nb$ if $\forall e\, (\varphi_e^g(e)\da \,\imp\, f(e) \neq \varphi_e^g(e))$.  A function $f \colon \Nb \to \Nb$ is said to have \emph{$\pa$-degree} relative to $g \colon \Nb \to \Nb$ if $f$ computes a function that is $\dnc_2$ relative to $g$.  The terminology is explained by the classical fact that $\dnc_2$ functions compute complete consistent extensions of Peano arithmetic and \emph{vice versa}.  If $A, B \subseteq \Nb$ are disjoint sets, then an \emph{$AB$-separator} is a set $S \subseteq \Nb$ such that $A \subseteq S \subseteq \Nb \setminus B$.  For any $g \colon \Nb \to \Nb$, let $A = \{e : \varphi_e^g(e) = 0\}$ and $B = \{e : \varphi_e^g(e) = 1\}$.  Then $A$ and $B$ are disjoint sets that are c.e.\ relative to $g$, and the (characteristic functions of the) $AB$-separators are exactly the functions that are $\dnc_2$ relative to $g$.  Thus a function has $\pa$-degree relative to $g$ if and only if it computes an $AB$-separator for this $A$ and $B$.

\subsection*{Computable structure theory notation and definitions}

Fix a computable language $\mf{L}$.  Following~\cites{MontalbanBookWithin,MontalbanBookBeyond}, a \emph{presentation} or \emph{copy} of a countable $\mf{L}$-structure $\mc{M}$ is an isomorphic copy $\mc{A}$ of $\mc{M}$ whose domain is a (non-empty) subset of $\Nb$.  In this case, $\mc{A}$ can be encoded as the object
\begin{align*}
\mc{P} \;=\; |\mc{A}| \;\oplus\; \bigoplus_{R \in \mf{L}_{\mathsf{rel}}} R^{\mc{A}} \;\oplus\; \bigoplus_{F \in \mf{L}_{\mathsf{func}}} F^{\mc{A}} \;\oplus\; \bigoplus_{R \in \mf{L}_{\mathsf{const}}} c^{\mc{A}}
\end{align*}
where $|\mc{A}| \subseteq \Nb$ is the domain of $\mc{A}$, and $\mf{L}_{\mathsf{rel}}$, $\mf{L}_{\mathsf{func}}$, and $\mf{L}_{\mathsf{const}}$ are the relation, function, and constant symbols of $\mf{L}$.

If $\mc{A}$ is an $\mf{L}$-structure with $|\mc{A}| \subseteq \Nb$, then we typically identify $\mc{A}$ with its encoding $\mc{P}$, which itself can be identified with a subset of the natural numbers using the $n$-tupling function. 
Thus we say that $\mc{A}$ is \emph{computable} if $\mc{P}$ is computable.  So a computable $\mf{L}$-structure consists of a non-empty computable domain and uniformly computable interpretations of the relation, function, and constant symbols of $\mf{L}$.  Similarly, a \emph{uniformly computable sequence of $\mf{L}$-structures} $(\mc{A}_n : n \in \Nb)$ consists of a uniformly computable sequence $(|\mc{A}_n| : n \in \Nb)$ of non-empty domains, along with uniformly computable interpretations of all the symbols of $\mf{L}$ in the structures $(\mc{A}_n : n \in \Nb)$. All of these concepts readily relativize.

For a countable structure $\mc{A}$ in a computable language, the \emph{degree spectrum} of $\mc{A}$, denoted $\dgsp(\mc{A})$, is the collection of Turing degrees that can compute a copy of $\mc{A}$.  If the degree spectrum of $\mc{A}$ has a least element, then that is called the \emph{degree} of $\mc{A}$.  So $\mc{A}$ has degree $\degT(X)$ if $\mc{A}$ has an $X$-computable presentation and every presentation of $\mc{A}$ computes $X$. If no such degree exists, we say that $\mc{A}$ has no degree.
It is a result of Richter \cite{Ric81} that if $\mc{L}$ is a linear order, then the only possible degree of $\mc{L}$ is $\mathbf{0}$.

\begin{Theorem}[\cite{Ric81}]\label{richter}
If $\mc{L}$ is a linear order, then $\dgsp(\mc{L})$ has a minimum element if and only if $\mc{L}$ has a computable copy.
In particular, $\mc{L}$ has a degree if and only if it has a computable copy.
\end{Theorem}

\subsection*{Cohesive powers and products}

Cohesive products and cohesive powers of computable structures are introduced by Dimitrov~\cite{DimitrovCohPow}.  We give the definition as it is presented in~\cite{CohPowJournal}.

\begin{Definition}\label{def-CohProd}
Let $\mf{L}$ be a computable language.  Let $(\mc{A}_n : n \in \Nb)$ be a uniformly computable sequence of $\mf{L}$-structures with corresponding uniformly computable sequence of non-empty domains $(|\mc{A}_n| : n \in \Nb)$.  Let $C \subseteq \Nb$ be a cohesive set.  The \emph{cohesive product of $(\mc{A}_n : n \in \Nb)$ over $C$} is the $\mf{L}$-structure $\prod_C \mc{A}_n$ defined as follows.

\begin{itemize}
\item Let $D$ be the set of partial computable functions $\varphi$ such that $\forall n \, (\varphi(n)\da \,\imp\, \varphi(n) \in |\mc{A}_n|)$ and $C \subseteq^* \dom(\varphi)$.

\medskip

\item For $\varphi, \psi \in D$, let $\varphi =_C \psi$ denote $C \subseteq^* \{n : \varphi(n)\da = \psi(n)\da\}$.  The relation $=_C$ is an equivalence relation on $D$.  Let $[\varphi]$ denote the equivalence class of $\varphi \in D$ with respect to $=_C$.

\medskip

\item The domain of $\prod_C \mc{A}_n$ is the set $|\prod_C \mc{A}_n| = \{[\varphi] : \varphi \in D\}$.

\medskip

\item 
Let $R$ be an $m$-ary relation symbol of $\mf{L}$.  For each $[\varphi_0], \dots, [\varphi_{m-1}] \in |\prod_C \mc{A}_n|$, define $R^{\prod_C \mc{A}_n}([\varphi_0], \dots, [\varphi_{m-1}])$ by
\begin{align*}
R^{\prod_C \mc{A}_n}([\varphi_0], \dots, [\varphi_{m-1}]) \;\;\Biimp\;\; C \subseteq^* \bigl\{n : R^{\mc{A}_n}(\varphi_0(n), \dots, \varphi_{m-1}(n))\bigr\}.
\end{align*}
Here, $R^{\mc{A}_n}(\varphi_0(n), \dots, \varphi_{m-1}(n))$ includes the condition that $\varphi_i(n)\da$ for each $i < m$.

\medskip

\item Let $f$ be an $m$-ary function symbol of $\mf{L}$.  For each $[\varphi_0], \dots, [\varphi_{m-1}] \in |\prod_C \mc{A}_n|$, let $\psi$ be the partial computable function defined by
\begin{align*}
\psi(n) \keq f^{\mc{A}_n}(\varphi_0(n), \dots, \varphi_{m-1}(n)),
\end{align*}
and notice that $C \subseteq^* \dom(\psi)$ because $C \subseteq^* \dom(\varphi_i)$ for each $i < m$.  Define $f^{\prod_C \mc{A}_n}$ by
\begin{align*}
f^{\prod_C \mc{A}_n}([\varphi_0], \dots, [\varphi_{m-1}]) = [\psi].
\end{align*}

\medskip

\item Let $c$ be a constant symbol of $\mf{L}$.  Let $\psi$ be the total computable function defined by $\psi(n) = c^{\mc{A}_n}$, and define $c^{\prod_C \mc{A}_n} = [\psi]$.
\end{itemize}

In the case where $\mc{A}_n$ is the same fixed computable structure $\mc{A}$ for every $n$, the cohesive product $\prod_C \mc{A}_n$ is called the \emph{cohesive power of $\mc{A}$ over $C$} and is denoted $\prod_C \mc{A}$.
\end{Definition}

We present versions of {\L}o\'{s}'s theorem for cohesive products and cohesive powers as in~\cite{CohPowJournal}.  The results of~\cite{CohPowJournal} also show that these versions of {\L}o\'{s}'s theorem are best possible.  Analogs of {\L}o\'{s}'s theorem for cohesive powers of arbitrary computable structures originate with Dimitrov's \emph{fundamental theorem of cohesive powers}~\cite{DimitrovCohPow}.  Analogs of {\L}o\'{s}'s theorem for various flavors of effective ultrapowers of the particular structure $\mc{N}$ extend back to work of Hirschfeld~\cite{HirschfeldModels} and of McLaughlin~\cite{McLaughlinSurvey}.  See also~\cite{ShavrukovPrimePowers}.

A formula $\Phi$ is \emph{uniformly decidable} in a uniformly computable sequence of $\mf{L}$-structures $(\mc{A}_n : n \in \Nb)$ if there is an algorithm that determines whether or not $\mc{A}_n \models \Psi(\vec{a})$ given $n$, a subformula $\Psi(\vec{v})$ of $\Phi$, and a sequence of elements $\vec{a}$ of $|\mc{A}_n|$ of the appropriate length.  
The presence of more uniform decidability extends the sorts of formulas that {\L}o\'{s}'s theorem for cohesive products and cohesive powers applies to.
We abuse terminology by saying that a formula is $\Delta_n$ if it is equivalent to both a $\Sigma_n$ formula and a $\Pi_n$ formula.

\begin{Lemma}[\cite{CohPowJournal}*{Lemma~2.5}]\label{lem-LosProdParamHelper}
Let $\mf{L}$ be a computable language, let $(\mc{A}_n : n \in \Nb)$ be a uniformly computable sequence of $\mf{L}$-structures, and let $C$ be a cohesive set.  Let $\Phi(\vec{x}, \vec{y}, v_0, \dots, v_{m-1})$ be a formula that is uniformly decidable in $(\mc{A}_n : n \in \Nb)$.

\begin{enumerate}[(1)]
\item\label{it-LosProdPramSig2Helper} For any $[\varphi_0], \dots, [\varphi_{m-1}] \in |\prod_C \mc{A}_n|$,
\begin{align*}
\prod\nolimits_C \mc{A}_n \models \exists \vec{x}\, \forall \vec{y}\, \Phi(\vec{x}, \vec{y}, [\varphi_0], \dots, [\varphi_{m-1}]) \;\;\Imp\;\; C \subseteq^* \bigl\{n : \mc{A}_n \models \exists \vec{x}\, \forall \vec{y}\, \Phi(\vec{x}, \vec{y}, \varphi_0(n), \dots, \varphi_{m-1}(n))\bigr\}.
\end{align*}

\medskip

\item\label{it-LosProdPramPi2Helper} For any $[\varphi_0], \dots, [\varphi_{m-1}] \in |\prod_C \mc{A}_n|$,
\begin{align*}
C \subseteq^* \bigl\{n : \mc{A}_n \models \forall \vec{x}\, \exists \vec{y}\, \Phi(\vec{x}, \vec{y}, \varphi_0(n), \dots, \varphi_{m-1}(n))\bigr\} \;\;\Imp\;\; \prod\nolimits_C \mc{A}_n \models \forall \vec{x}\, \exists \vec{y}\, \Phi(\vec{x}, \vec{y}, [\varphi_0], \dots, [\varphi_{m-1}]).
\end{align*}
\end{enumerate}
\end{Lemma}

\begin{Theorem}[\cite{CohPowJournal}*{Theorem~2.7}]\label{thm-LosProdParam}
Let $\mf{L}$ be a computable language, let $(\mc{A}_n : n \in \Nb)$ be a uniformly computable sequence of $\mf{L}$-structures, and let $C$ be a cohesive set.
\begin{enumerate}[(1)]
\item\label{it-LosProdPramSig2} Let $\Phi(v_0, \dots, v_{m-1})$ be a $\Sigma_2$ formula.  Then for any $[\varphi_0], \dots, [\varphi_{m-1}] \in |\prod_C \mc{A}_n|$,
\begin{align*}
\prod\nolimits_C \mc{A}_n \models \Phi([\varphi_0], \dots, [\varphi_{m-1}]) \;\;\Imp\;\; C \subseteq^* \bigl\{n : \mc{A}_n \models \Phi(\varphi_0(n), \dots, \varphi_{m-1}(n))\bigr\}.
\end{align*}

\medskip

\item\label{it-LosProdPramPi2} Let $\Phi(v_0, \dots, v_{m-1})$ be a $\Pi_2$ formula.  Then for any $[\varphi_0], \dots, [\varphi_{m-1}] \in |\prod_C \mc{A}_n|$,
\begin{align*}
C \subseteq^* \bigl\{n : \mc{A}_n \models \Phi(\varphi_0(n), \dots, \varphi_{m-1}(n))\bigr\} \;\;\Imp\;\; \prod\nolimits_C \mc{A}_n \models \Phi([\varphi_0], \dots, [\varphi_{m-1}]).
\end{align*}

\medskip

\item\label{it-LosProdPramDelta2} Let $\Phi(v_0, \dots, v_{m-1})$ be a $\Delta_2$ formula.  Then for any $[\varphi_0], \dots, [\varphi_{m-1}] \in |\prod_C \mc{A}_n|$,
\begin{align*}
\prod\nolimits_C \mc{A}_n \models \Phi([\varphi_0], \dots, [\varphi_{m-1}]) \;\;\Biimp\;\; C \subseteq^* \bigl\{n : \mc{A}_n \models \Phi(\varphi_0(n), \dots, \varphi_{m-1}(n))\bigr\}.
\end{align*}
\end{enumerate}
\end{Theorem}

\begin{Theorem}[\cite{CohPowJournal}*{Theorem~2.9}]\label{thm-LosGeneral}
Let $\mf{L}$ be a computable language, let $\mc{A}$ be a computable $\mf{L}$-structure, and let $C$ be a cohesive set.

\begin{enumerate}[(1)]
\item\label{it-LosDelta3Sent} Let $\Phi$ be a $\Delta_3$ sentence.  Then $\mc{A} \models \Phi$ if and only if $\prod_C \mc{A} \models \Phi$.

\medskip

\item\label{it-LosSigma3Sent} Let $\Phi$ be a $\Sigma_3$ sentence.  If $\mc{A} \models \Phi$, then $\prod_C \mc{A} \models \Phi$.
\end{enumerate}
\end{Theorem}

By \cref{thm-LosGeneral}, a cohesive product of (simple) graphs is again a graph, and a cohesive product of linear orders is again a linear order.  We need \cref{thm-IsoGenSum} below in \Cref{sec-lo}, which is a slight generalization of~\cite{CohPowJournal}*{Theorem~6.3} explaining how cohesive products commute with generalized sums of linear orders.  We include the proof for completeness.

\begin{Definition}[see~\cite{RosBook}*{Definition~1.38}]\label{def-GenSum}
Let $\mc{K}$ be a linear order, and let $(\mc{M}_k : k \in |\mc{K}|)$ be a sequence of linear orders indexed by $|\mc{K}|$.  The \emph{generalized sum} $\sum_{k \in |\mc{K}|}\mc{M}_k$ of $(\mc{M}_k : k \in |\mc{K}|)$ over $\mc{K}$ is the linear order $\mc{S} = (S; \prec_\mc{S})$ defined as follows.  Write $\mc{K} = (K; \prec_\mc{K})$, and write $\mc{M}_k = (M_k; \prec_{\mc{M}_k})$ for each $k \in K$.  Define $S = \{(k, m) : k \in K \land m \in M_k\}$, and define
\begin{align*}
(k_0, m_0) \prec_{\mc{S}} (k_1, m_1) \quad\text{if and only if}\quad (k_0 \prec_\mc{K} k_1) \;\lor\; (k_0 = k_1 \,\land\, m_0 \prec_{\mc{M}_{k_0}} m_1).
\end{align*}
\end{Definition}

Note that if $\mc{K}$ is computable and $(\mc{M}_k : k \in |\mc{K}|)$ is uniformly computable, then $\sum_{k \in |\mc{K}|}\mc{M}_k$ is also computable.  \Cref{thm-IsoGenSum} says that if $(\mc{K}_n : n \in \Nb)$ and $(\mc{M}_{n,k} : n \in \Nb, k \in |\mc{K}_n|)$ are uniformly computable sequences of linear orders and $C$ is cohesive, then
\begin{align*}
\prod\nolimits_C \sum_{k \in |\mc{K}_n|}\mc{M}_{n,k} \quad\iso\quad \sum_{[\theta] \in \left|\prod_C \mc{K}_n\right|}\prod\nolimits_C \mc{M}_{n,\theta(n)}.
\end{align*}
The left-hand side is a cohesive product of generalized sums, where the $n$\textsuperscript{th} sum is of $(\mc{M}_{n,k} : k \in |\mc{K}_n|)$ over $\mc{K}_n$.  The right-hand side is a generalized sum of cohesive products over a linear order that is itself a cohesive product and thus requires further explanation.  The sum is over the linear order $\prod_C \mc{K}_n$.  Term $[\theta]$ in the sum is the cohesive product $\prod_C \mc{M}_{n,\theta(n)}$ of the sequence $\mc{M}_{0,\theta(0)}, \mc{M}_{1,\theta(1)}, \mc{M}_{2,\theta(2)}, \dots$.  Of course, $\theta(n)$ and hence $\mc{M}_{n,\theta(n)}$ may be undefined for some $n$, but $C \subseteq^* \dom(\theta)$ by virtue of $[\theta] \in \left|\prod_C \mc{K}_n\right|$, so $\mc{M}_{n,\theta(n)}$ is defined for almost every $n \in C$.  Thus by $\prod_C \mc{M}_{n,\theta(n)}$, we mean the product constructed from partial computable functions $\varphi$ with $C \subseteq^* \dom(\varphi) \subseteq \dom(\theta)$ and $\forall n\, (\varphi(n)\da \,\imp\, \varphi(n) \in |\mc{M}_{n,\theta(n)}|)$.  See~\cite{CohPowJournal}*{Section~6} for further details.

\begin{Theorem}[Generalizing~\cite{CohPowJournal}*{Theorem~6.3}]\label{thm-IsoGenSum}
Let $(\mc{K}_n : n \in \Nb)$ and $(\mc{M}_{n,k} : n \in \Nb, k \in |\mc{K}_n|)$ be uniformly computable sequences of linear orders.  Let $C$ be a cohesive set.  Then
\begin{align*}
\prod\nolimits_C \sum_{k \in |\mc{K}_n|}\mc{M}_{n,k} \quad\iso\quad \sum_{[\theta] \in \left|\prod_C \mc{K}_n\right|}\prod\nolimits_C \mc{M}_{n,\theta(n)}.
\end{align*}
\end{Theorem}

\begin{proof}
To ease notation, let
\begin{align*}
\mc{L}_n &= \sum_{k \in |\mc{K}_n|}\mc{M}_{n,k} & \text{for each $n$,}\\ \\
\mc{X} &= \prod\nolimits_C \mc{K}_n,\\ \\
\mc{Y}_{[\theta]_\mc{X}} &= \prod\nolimits_C \mc{M}_{n, \theta(n)} & \text{for each $[\theta]_\mc{X} \in |\mc{X}|$,}\\ \\
\mc{A} &= \prod\nolimits_C \mc{L}_n,\\ \\
\mc{B} &= \sum_{[\theta]_\mc{X} \in |\mc{X}|} \mc{Y}_{[\theta]_\mc{X}}.
\end{align*}
The goal is to show that $\mc{A} \iso \mc{B}$.  The elements of $\mc{A}$ are of the form $[\varphi]_\mc{A}$ for partial computable $\varphi$ with $\forall n \, (\varphi(n)\da \,\imp\, \varphi(n) \in |\mc{L}_n|)$ and $C \subseteq^* \dom(\varphi)$.  The elements of $\mc{B}$ are of the form $\bigl([\theta]_\mc{X}, [\tau]_{\mc{Y}_{[\theta]_\mc{X}}}\bigr)$ for partial computable $\theta$ and $\tau$ with $\forall n \, (\theta(n)\da \,\imp\, \theta(n) \in |\mc{K}_n|)$, $C \subseteq^* \dom(\tau) \subseteq \dom(\theta)$, and $\forall n \, (\tau(n)\da \,\imp\, \tau(n) \in |\mc{M}_{n, \theta(n)}|)$.

Define a function $F \colon |\mc{A}| \imp |\mc{B}|$ as follows.  For $[\varphi]_\mc{A} \in |\mc{A}|$, we have that $\varphi(n) \in |\mc{L}_n|$ and therefore that $\varphi(n) = \la k, m \ra$ for some $k \in |\mc{K}_n|$ and $m \in | \mc{M}_{n,k} |$ whenever $\varphi(n)\da$.  Let $\theta = \pi_0 \circ \varphi$, and let $\tau = \pi_1 \circ \varphi$.  Then $[\theta]_\mc{X} \in |\mc{X}|$ and $[\tau]_{\mc{Y}_{[\theta]_\mc{X}}} \in |\mc{Y}_{[\theta]_\mc{X}}|$.  Set $F([\varphi]_\mc{A}) = \bigl([\theta]_\mc{X}, [\tau]_{\mc{Y}_{[\theta]_\mc{X}}}\bigr)$.  To see that $F$ is well-defined, observe that if $\varphi =_C \psi$, then also $\pi_0 \circ \varphi =_C \pi_0 \circ \psi$ and $\pi_1 \circ \varphi =_C \pi_1 \circ \psi$.

To show that $F$ is an isomorphism between the linear orders $\mc{A}$ and $\mc{B}$, it suffices to show that $F$ is surjective and order-preserving.

For surjectivity, consider an element $\bigl([\theta]_\mc{X}, [\tau]_{\mc{Y}_{[\theta]_\mc{X}}}\bigr)$ of $\mc{B}$.  Define a partial computable $\varphi$ by $\varphi(n) \keq \la \theta(n), \tau(n) \ra$.  Then $\varphi(n) \in |\mc{L}_n|$ whenever $\varphi(n)\da$, and $C \subseteq^* \dom(\varphi)$ because $C \subseteq^* \dom(\tau) \subseteq \dom(\theta)$.  Therefore $[\varphi]_\mc{A} \in |\mc{A}|$ and $F([\varphi]_\mc{A}) = \bigl([\theta]_\mc{X}, [\tau]_{\mc{Y}_{[\theta]_\mc{X}}}\bigr)$.

For order-preserving, suppose that $[\varphi]_\mc{A}$ and $[\psi]_\mc{A}$ are members of $\mc{A}$ with $[\varphi]_\mc{A} \prec_\mc{A} [\psi]_\mc{A}$.  Then $(\forae n \in C)(\varphi(n) \prec_{\mc{L}_n} \psi(n))$.  Write $\theta = \pi_0 \circ \varphi$, $\tau = \pi_1 \circ \varphi$, $\alpha = \pi_0 \circ \psi$, and $\beta = \pi_1 \circ \psi$.  Then
\begin{align*}
(\forae n \in C)\Bigl( \bigl( \theta(n) \prec_{\mc{K}_n} \alpha(n) \bigr) \;\lor\; \bigl( \theta(n) = \alpha(n) \,\land\, \tau(n) \prec_{\mc{M}_{n, \theta(n)}} \beta(n) \bigr) \Bigr)
\end{align*}
By cohesiveness, either
\begin{itemize}
\item $(\forae n \in C)\bigl( \theta(n) \prec_{\mc{K}_n} \alpha(n) \bigr)$ or

\smallskip

\item $(\forae n \in C)\bigl( \theta(n) = \alpha(n) \,\land\, \tau(n) \prec_{\mc{M}_{n, \theta(n)}} \beta(n) \bigr)$.
\end{itemize}
In the first case, $[\theta]_\mc{X} \prec_\mc{X} [\alpha]_\mc{X}$.  In the second case, $[\theta]_\mc{X} = [\alpha]_\mc{X}$ and $[\tau]_{\mc{Y}_{[\theta]_\mc{X}}} \prec_{\mc{Y}_{[\theta]_\mc{X}}} [\beta]_{\mc{Y}_{[\theta]_\mc{X}}}$.  Thus in either case,
\begin{align*}
F([\varphi]_\mc{A}) = \bigl( [\theta]_\mc{X}, [\tau]_{\mc{Y}_{[\theta]_\mc{X}}} \bigr) \prec_\mc{B} \bigl( [\alpha]_\mc{X}, [\beta]_{\mc{Y}_{[\alpha]_\mc{X}}} \bigr) = F([\psi]_\mc{A}),
\end{align*}
as desired.
\end{proof}

\section{Encoding information into cohesive powers}\label{sec-graph}

We first show that there are computable structures whose cohesive powers are capable of encoding arbitrary information via the cohesive set.

\begin{Proposition}\label{prop-EncodeAll}
There is a fixed computable structure $\mc{A}$ such that for every set $X \subseteq \Nb$, there is a cohesive set $C$ such that every presentation of $\prod_C \mc{A}$ computes $X$.
\end{Proposition}

\begin{proof}
The language consists of a computable sequence of unary predicates $(U_\sigma : \sigma \in 2^{< \Nb})$, which we think of as names for sets.  Compute a structure $\mc{A} = \bigl(\Nb; (U_\sigma : \sigma \in 2^{< \Nb})\bigr)$ as follows.   Let $U_\emptyset = \Nb$.  Given $U_\sigma$ enumerated in order as $U_\sigma = \{p_0 < p_1 < p_2 < \cdots\}$, let $U_{\sigma^\smf 0} = \{p_{2n} : n \in \Nb\}$ and $U_{\sigma^\smf 1} = \{p_{2n+1} : n \in \Nb\}$.  If $\sigma \sqsubseteq \tau$, then $U_\sigma \supseteq U_\tau$; and if $\sigma$ and $\tau$ are incomparable, then $U_\sigma \cap U_\tau  = \emptyset$.

Let $X \subseteq \Nb$ be any set.  Choose an increasing sequence $c_0 < c_1 < c_2 < \cdots$ with $c_n \in U_{X \rst n}$ for each $n$.  Let $C$ be any cohesive subset of $\{c_n : n \in \Nb\}$.  Then for any $\sigma \in 2^{< \Nb}$, we have that $C \subseteq^* U_\sigma$ if and only if $\sigma \sqsubseteq X$.  Let $f \colon \Nb \to \Nb$ be the computable function $f(n) = n$.  For any $\sigma \in 2^{< \Nb}$,
\begin{align*}
U_\sigma^{\prod_C \mc{A}}([f]) \;\;\Biimp\;\; C \subseteq^* \{n : U_\sigma(f(n))\} \;\;\Biimp\;\; C \subseteq^* \{n : U_\sigma(n)\}
\;\;\Biimp\;\; C \subseteq^* U_\sigma \;\;\Biimp\;\; \sigma \sqsubseteq X.
\end{align*}
Every presentation $\mc{P}$ of $\prod_C \mc{A}$ computes $X$ because if $a \in |\mc{P}|$ is the element corresponding to $[f]$, then the $\mc{P}$-computable set $\{\sigma \in 2^{<\Nb} : U_\sigma^\mc{P}(a)\}$ consists of exactly the initial segments of $X$.
\end{proof}

The use of an infinite language in \cref{prop-EncodeAll} is just for expedience.  With a little more coding, the same effect may be achieved with a finite language.

In light of \cref{prop-EncodeAll}, if cohesive products are to be thought of as effectivized ultraproducts, then some effectivity assumption should be imposed on the cohesive sets over which the products are taken.  The next proposition shows that if we restrict to $\Delta_k$ ($k \geq 2$) cohesive sets, then the cohesive products over these cohesive sets always have $\Delta_{k+1}$ presentations.

\begin{Proposition}\label{prop-Delta3Pres}
Let $\mf{L}$ be a computable language, and let $(\mc{A}_n : n \in \Nb)$ be a uniformly computable sequence of $\mf{L}$-structures.  Let $C$ be a $\Delta_k$ cohesive set for some $k \geq 2$.  Then $\prod_C \mc{A}_n$ has a $\Delta_{k+1}$ presentation.
\end{Proposition}

\begin{proof}
We describe a $\Delta_{k+1}$ presentation $\mc{P}$ of $\prod_C \mc{A}_n$.  The elements of $\prod_C \mc{A}_n$ are equivalence classes $[\varphi]$ of partial computable functions $\varphi \colon \Nb \to \Nb$ where $\forall n \, (\varphi(n)\da \,\imp\, \varphi(n) \in |\mc{A}_n|)$ and $C \subseteq^* \dom(\varphi)$.
Let
\begin{align*}
D = \{\la e, N \ra : \forall n \, (\varphi_e(n)\da \,\imp\, \varphi_e(n) \in |\mc{A}_n|) \;\land\; (\forall n > N)(n \in C \,\imp\, \varphi_e(n)\da)\}.
\end{align*}
The set $D$ is $\Pi_k$ because $C$ is $\Delta_k$ (and $k \geq 2$).  The elements of $D$ are intended to represent the elements of $\prod_C \mc{A}_n$.  Thus we must identify when different elements of $D$ represent the same element of $\prod_C \mc{A}_n$.
Define
\begin{align*}
\la e, N \ra \sim \la i, M \ra \quad\equiv\quad \la e, N \ra \in D \,\land\, \la i, M \ra \in D \,\land\, (\exists K)(\forall n > K)(n \in C \,\imp\, \varphi_e(n) = \varphi_i(n))
\end{align*}
As written, $\la e, N \ra \sim \la i, M \ra$ is a $\Sigma_{k+1}$ relation because $C$ is $\Delta_k$ and $D$ is $\Pi_k$.  However, if $\la e, N \ra \in D$ and $\la i, M \ra \in D$ but $\neg(\exists K)(\forall n > K)(n \in C \,\imp\, \varphi_e(n) = \varphi_i(n))$, then by cohesiveness $(\exists K)(\forall n > K)(n \in C \,\imp\, \varphi_e(n) \neq \varphi_i(n))$.  Therefore
\begin{align*}
\la e, N \ra \nsim \la i, M \ra \quad\equiv\quad \la e, N \ra \notin D \,\lor\, \la i, M \ra \notin D \,\lor\, (\exists K)(\forall n > K)(n \in C \,\imp\, \varphi_e(n) \neq \varphi_i(n)),
\end{align*}
which is also $\Sigma_{k+1}$.  Thus $\la e, N \ra \sim \la i, M \ra$ is a $\Delta_{k+1}$ relation.

Let $X$ be the set of least representatives:
\begin{align*}
X = \Bigl\{\la e, N \ra : \la e, N \ra \in D \,\land\, \bigl(\forall \la i, M \ra < \la e, N \ra\bigr)\bigl(\la e, N \ra \nsim \la i, M \ra\bigr)\Bigr\}.
\end{align*}
Then $X$ is $\Delta_{k+1}$, and we take it to be the domain of our presentation $\mc{P}$.

For each $m$-ary relation symbol $R$, define
\begin{align*}
R^{\mc P}\bigl(\la e_0, N_0 \ra,& \dots, \la e_{m-1}, N_{m-1} \ra\bigr) \;\equiv\\
&(\forall i < m)\bigl(\la e_i, N_i \ra \in X\bigr) \,\land\, (\exists K)(\forall n > K)\bigl(n \in C \,\imp\, R^{\mc{A}_n}(\varphi_{e_0}(n), \dots, \varphi_{e_{m-1}}(n))\bigr).
\end{align*}
By cohesiveness and the same analysis as for the $\la e, N \ra \sim \la i, M \ra$ relation, $R^{\mc P}$ is a $\Delta_{k+1}$ relation on $X^m$ uniformly in $R$.

For each $m$-ary function symbol $F$, allowing $m=0$ to include the constant symbols, define
\begin{align*}
F^{\mc P}\bigl(&\la e_0, N_0 \ra, \dots, \la e_{m-1}, N_{m-1} \ra\bigr) = \la e_m, N_m \ra \;\equiv\\
&(\forall i \leq m)\bigl(\la e_i, N_i \ra \in X\bigr) \,\land\, (\exists K)(\forall n > K)\bigl(n \in C \,\imp\, \varphi_{e_m}(n) = F^{\mc{A}_n}(\varphi_{e_0}(n), \dots, \varphi_{e_{m-1}}(n))\bigr)
\end{align*}
Again by cohesiveness, $F^{\mc P}$ is a $\Delta_{k+1}$ function $X^m \to X$ uniformly in $F$.

We have defined a $\Delta_{k+1}$ $\mf{L}$-structure $\mc{P}$ with domain $X$.  An element $\la e, N \ra \in X$ corresponds to the element $[\varphi_e]$ of $\prod_C \mc{A}_n$.  Conversely, an element $[\varphi]$ of $\prod_C \mc{A}_n$ corresponds to the element $\la e, N \ra$ of $X$, where $\la e, N \ra$ is the least pair in $D$ such that $(\exists K)(\forall n > K)(n \in C \,\imp\, \varphi_e(n) = \varphi(n))$.  It is straightforward to check that this correspondence is bijective and gives an isomorphism between $\mc{P}$ and $\prod_C \mc{A}_n$, so $\mc{P}$ is a $\Delta_{k+1}$ presentation of $\prod_C \mc{A}_n$.
\end{proof}

Now we compute a graph $\mc{G}$ for which every presentation of every cohesive power computes $0''$.  Thus if we restrict to $\Delta_2$ cohesive sets, the cohesive powers of $\mc{G}$ have maximum complexity.  If $C$ is a $\Delta_2$ cohesive set, then $\prod_C \mc{G}$ has a $0''$-computable presentation, and every presentation of $\prod_C \mc{G}$ computes $0''$.

\begin{Theorem}\label{thm-MaxGraphPow}
There is a computable graph $\mc{G}$ such that for every cohesive set $C$, every presentation of the cohesive power $\prod_C \mc{G}$ computes $0''$.
\end{Theorem}

\begin{proof}
Given a $\Sigma_3$ set $X$, the plan is to compute a graph $\mc{G}$ such that for every cohesive set $C$, the set $X$ is c.e.\ in every presentation of $\prod_C \mc{G}$.  If we take $X$ to be the particular $\Sigma_3$ set $0'' \oplus (\Nb \setminus 0'')$, then both $0''$ and its complement are c.e.\ and hence computable in every presentation of $\prod_C \mc{G}$ for every cohesive set $C$.

Let $X = \{k : \Phi(k)\}$ be a $\Sigma_3$ set, where $\Phi(k) \equiv \exists \ell\, \forall i\, \exists j\, R(k,\ell,i,j)$ for the computable relation $R$.  We compute a graph $\mc{G}$ on domain $|\mc{G}| = A \cup B$, where $A = \{a_n : n \in \Nb\}$ and $B = \{b_n : n \in \Nb\}$.  For specificity, say $a_n = 2n$ and $b_n = 2n+1$.  The graph $\mc{G}$ consists of cycles of various lengths having exactly one vertex from $A$ and the rest from $B$ (plus some vertices in $B$ may be isolated).  Each vertex in $A$ may be in many cycles, but each vertex in $B$ is in at most one cycle.  This will be done in such a way so that if $f \colon \Nb \to \Nb$ is the computable function $f(n) = a_n$, then, for any cohesive set $C$, $\Phi(k)$ holds if and only if, in $\prod_C \mc{G}$, there is an $\ell$ and a cycle of length $\la k, \ell \ra + 3$ containing $[f]$.  This makes $X$ c.e.\ in every presentation of $\prod_C \mc{G}$ by enumerating the $k$ for which there is a cycle of length $\la k, \ell \ra + 3$ for some $\ell$ that contains the element corresponding to $[f]$.

Compute $\mc{G}$ in stages.  At stage $s$, consider the triples $\la n, k, \ell \ra < s$.  For each such triple, if $(\forall i < n)(\exists j < s)\, R(k,\ell,i,j)$ and $a_n$ does not already belong to a cycle of length $\la k, \ell \ra + 3$, then add a cycle of length $\la k, \ell \ra + 3$ to $\mc{G}$ consisting of $a_n$ and $\la k, \ell \ra + 2$ fresh elements $>\! s$ from $B$.  This completes the computation of $\mc{G}$.

Let $C$ be a cohesive set, and let $f \colon \Nb \to \Nb$ be the computable function $f(n) = a_n$.  Given $k$, we show that $\Phi(k)$ holds if and only if there is an $\ell$ and a cycle of length $\la k, \ell \ra + 3$ in $\prod_C \mc{G}$ containing $[f]$.

Suppose that $\Phi(k)$ holds, and let $\ell$ witness $\forall i\, \exists j\, R(k,\ell,i,j)$.  Then $(\forall i < n)(\exists j < s)\, R(k,\ell,i,j)$ holds for every $n$ and every sufficiently large $s$.  So for every $n$, $\mc{G}$ has a cycle of length $\la k, \ell \ra + 3$ containing $a_n$.  The formula $\Psi_{k,\ell}(v) \equiv \text{``there is a cycle of length $\la k, \ell \ra + 3$ containing $v$''}$ is $\Sigma_1$, and $\mc{G} \models \Psi_{k,\ell}(f(n))$ for every $n$.  Therefore $\prod_C \mc{G} \models \Psi_{k,\ell}([f])$ by \cref{thm-LosProdParam}.  That is, $\prod_C \mc{G}$ has a cycle of length $\la k, \ell \ra + 3$ containing $[f]$.

Suppose instead that $\Phi(k)$ fails, and consider an $\ell$.  This $\ell$ does not witness $\forall i\, \exists j\, R(k,\ell,i,j)$, so there is an $n_\ell$ such that $\forall j\, \neg R(k,\ell,n_\ell,j)$.  Thus $(\forall i < n)(\exists j < s)\, R(k,\ell,i,j)$ fails for every $n > n_\ell$ and every $s$.  Therefore, if $n > n_\ell$, then $\mc{G}$ does not have a cycle of length $\la k, \ell \ra + 3$ containing $a_n$.  The formula $\neg \Psi_{k,\ell}(v)$ is $\Pi_1$, and $\mc{G} \models \neg\Psi_{k,\ell}(f(n))$ for all $n > n_\ell$.  Therefore $\prod_C \mc{G} \models \neg\Psi_{k,\ell}([f])$ by \cref{thm-LosProdParam}, so $\prod_C \mc{G}$ does not have a cycle of length $\la k, \ell \ra + 3$ containing $[f]$.  This analysis was for an arbitrary $\ell$, thus for no $\ell$ does $\prod_C \mc{G}$ have a cycle of length $\la k, \ell \ra + 3$ containing $[f]$.
\end{proof}

\begin{Corollary}\label{cor-degDJ}
There is a computable graph $\mc{G}$ such that every cohesive power $\prod_C \mc{G}$ of $\mc{G}$ by a $\Delta_2$ cohesive set $C$ has degree $0''$.
\end{Corollary}

\begin{proof}
Let $\mc{G}$ be the computable graph from \cref{thm-MaxGraphPow}, and let $C$ be a $\Delta_2$ cohesive set.  Then $\prod_C \mc{G}$ has a $\Delta_3$ presentation (i.e., a $0''$-computable presentation) by \cref{prop-Delta3Pres}, and every presentation of $\prod_C \mc{G}$ computes $0''$ by \cref{thm-MaxGraphPow}.
\end{proof}

Note that this construction is not possible with a linear order by~\cref{richter}.
The next section is dedicated to seeing what is possible with a linear order.

\section{A linear order exhibiting a Tennenbaum-like phenomenon}\label{sec-lo}

We compute a linear order $\mc{L}$ with the following property.  If $C$ is a cohesive set and $\mc{P}$ is a presentation of $\prod_C \mc{L}$, then $\mc{P}''$ has $\pa$-degree relative to $0''$.  Such a $\mc{P}$ cannot be computable because $0''$ does not have $\pa$-degree relative to itself.  Therefore $\mc{L}$ exhibits a Tennenbaum-like phenomenon, as it is a computable linear order whose cohesive powers do not have computable presentations.  This answers our Question~\ref{q-lo} from the introduction.

To compute $\mc{L}$, we first compute a sequence $(\mc{L}_n : n \in \Nb)$ of linear orders where $\prod_C \mc{L}_n$ has the desired property for every cohesive set $C$:  the double-jump of every presentation of $\prod_C \mc{L}_n$ has $\pa$-degree relative to $0''$.  Then we show that we can essentially take the desired $\mc{L}$ to be the sum of the linear orders $(\mc{L}_n : n \in \Nb)$.

\begin{Theorem}\label{thm-ProdNonRec}
There is a uniformly computable sequence $(\mc{L}_n : n \in \Nb)$ of linear orders such that for every cohesive set $C$ and every presentation $\mc{P}$ of $\prod_C \mc{L}_n$, the double-jump $\mc{P}''$ has $\pa$-degree relative to $0''$.  Thus $\prod_C \mc{L}_n$ has no computable presentation.
\end{Theorem}

\begin{proof}
Let $A$ and $B$ be two disjoint $\Sigma_3$ sets.  We compute $(\mc{L}_n : n \in \Nb)$ so that whenever $C$ is a cohesive set and $\mc{P}$ is a presentation of $\prod_C \mc{L}_n$, there is an $AB$-separator computable from $\mc{P}''$.  The desired result follows because there are disjoint $\Sigma_3$ sets $A$ and $B$ whose separators are exactly the functions that are $\dnc_2$ relative to $0''$.

Given any $\Sigma_3$ set $A$, we can compute a sequence of sets $(A_k)_{k \in \Nb}$ such that, for all $k$, $k \in A$ if and only if $A_k$ has an infinite column: $(\exists i)(\forall s_0)(\exists s > s_0)(A_k(\la i, s \ra) = 1)$. 
For our disjoint $\Sigma_3$ sets $A$ and $B$, we therefore have two uniformly computable sequences of sets $(A_k)_{k \in \Nb}$ and $(B_k)_{k \in \Nb}$ where, for every $k$, at most one of $A_k$ and $B_k$ has an infinite column.

We uniformly compute linear orders $(\mc{L}_n : n \in \Nb)$ so that for every cohesive set $C$, the cohesive product $\prod_C \mc{L}_n$ has the form
\begin{align*}
\prod\nolimits_C \mc{L}_n \quad=\quad \sum_{k \in \Nb} (\mc{S}_k + \mc{R}_k) + \mc{J} \quad=\quad (\mc{S}_0 + \mc{R}_0 + \mc{S}_1 + \mc{R}_1 + \cdots) + \mc{J}
\end{align*}
with the following properties.

\begin{itemize}
\item For each $k$, $\mc{S}_k \;\iso\; (\bm{k+2}) + \eta + (\bm{k+2}) + \eta + (\bm{k+2})$.  The orders $\mc{S}_k$ and $\mc{S}_{k+1}$ serve as delimiters around the order $\mc{R}_k$.

\smallskip

\item For each $k$, $\mc{R}_k$ has no non-trivial dense interval.  In fact, every element of $\mc{R}_k$ other than the maximum element (should it exist) has a successor.

\smallskip

\item $\mc{J}$ does not have finite blocks of size $2$ or more.
\end{itemize}
It follows that for each $k$, $\mc{S}_k$ is the unique interval of $\prod_C \mc{L}_n$ with order-type $(\bm{k+2}) + \eta + (\bm{k+2}) + \eta + (\bm{k+2})$.  Furthermore, if $\mc{P}$ is a presentation of $\prod_C \mc{L}_n$, then $\mc{P}''$ can identify $\mc{S}_k$ by searching for elements $u_0, \dots, u_{k+1}, v_0, \dots, v_{k+1}, w_0, \dots, w_{k+1}$ such that the following two-quantifier properties hold:
\begin{itemize}
\item $u_0 \prec_{\mc P} \dots \prec_{\mc P} u_{k+1} \prec_{\mc P} v_0 \prec_{\mc P} \dots \prec_{\mc P} v_{k+1} \prec_{\mc P} w_0 \prec_{\mc P} \dots \prec_{\mc P} w_{k+1}$,
\item for each $i \leq k$, the intervals $(u_i, u_{i+1})_{\mc P}$, $(v_i, v_{i+1})_{\mc P}$, and $(w_i, w_{i+1})_{\mc P}$ are empty, and
\item the intervals $[u_{k+1}, v_0]_{\mc P}$ and $[v_{k+1}, w_0]_{\mc P}$ are dense.
\end{itemize}
The encoding plan is to arrange for $\mc{R}_k$ to have a maximum element if $A_k$ has an infinite column (i.e., if $k \in A$) and for $\mc{R}_k$ to have no maximum element if $B_k$ has an infinite column (i.e., if $k \in B$).  As we are only trying to compute a separator for $A$ and $B$, we will not need any control over whether $\mc{R}_k$ will have a maximum element in the case that neither $A_k$ nor $B_k$ has an infinite column.  Again, let $\mc{P}$ be a presentation of $\prod_C \mc{L}_n$.  To decode, use $\mc{P}''$ to compute the set $X = \{k : \text{$\mc{R}_k$ has a maximum element}\}$, which is an $AB$-separator.  To do this, given $k$, use $\mc{P}''$ to identify the intervals $\mc{S}_k$ and $\mc{S}_{k+1}$ as explained above.  Specifically, identify the right endpoint $w$ of $\mc{S}_k$ and the left endpoint $u$ of $S_{k+1}$ so that $\mc{R}_k$ is the interval $(w,u)$ of $\mc{P}$.  Then we use $\mc{P}''$ again to determine whether $\mc{R}_k$ has a maximum element and thus whether $k \in X$.  The idea of using orders like the $\mc{S}_k$'s to delimit diagonalization or, in this case, encoding regions like the $\mc{R}_k$'s comes from~\cite{JockuschSoareLO}.  The `$\mc{J}$' at the end of $\prod_C \mc{L}_n$ stands for `junk.'  We do not use this region for encoding or decoding, but we ensure that it does not contain finite blocks of size $\geq\! 2$ to ensure that $\prod_C \mc{L}_n$ does not contain unintended intervals of the form $\mc{S}_k$.

To implement the encoding, we uniformly compute linear orders $(\mc{M}_{n,k} : n, k \in \Nb)$ and take $\mc{L}_n = \sum_{k \in \Nb}\mc{M}_{n,k}$ for each $n$.  In fact, we compute $(\mc{M}_{n,k} : n, k \in \Nb)$ in such a way that the successor relation on $\mc{M}_{n,k}$ is c.e.\ and hence computable uniformly in $n$ and $k$.  Throughout we use `$\Nb$' to also denote the usual presentation of $(\Nb; <)$.  

For any cohesive set $C$, we have that
\begin{align*}
\prod\nolimits_C \mc{L}_n \quad=\quad \prod\nolimits_{C} \sum_{k \in \Nb}\mc{M}_{n,k} \quad\iso\quad \sum_{[\theta] \in \left|\prod_C \Nb \right|}\prod\nolimits_C \mc{M}_{n,\theta(n)}
\end{align*}
by \cref{thm-IsoGenSum}.  We also have that $\prod_C \Nb \iso \omega + \zeta\cdot\eta$ by~\cite{CohPowJournal}*{Theorem~4.5}.
Using the terminology of~\cite{CohPowJournal}, we call the initial $\omega$ of $\prod_C \Nb$ the \emph{standard part} and denote it $\left|\prod_C \Nb\right|_\std$, and we call the remaining $\zeta\cdot\eta$ within $\prod_C \Nb$ the \emph{non-standard part} and denote it $\left|\prod_C \Nb\right|_\nonstd$.  An element $[\theta]$ of $\prod_C \Nb$ is standard if and only if it is eventually constant on $C$, and $[\theta]$ is non-standard if and only if $\lim_{n \in C} \theta(n) = \infty$ (\cite{CohPowJournal}*{Lemma~4.2}).  In particular, the elements of $\left|\prod_C \Nb\right|_\std$ are exactly those represented by constant functions.  Therefore
\begin{align*}
\sum_{[\theta] \in \left|\prod_C \Nb \right|}\prod\nolimits_C \mc{M}_{n,\theta(n)} \quad&=\quad \sum_{[\theta] \in \left|\prod_C \Nb \right|_\std}\prod\nolimits_C \mc{M}_{n,\theta(n)} \;+\; \sum_{[\theta] \in \left|\prod_C \Nb \right|_\nonstd}\prod\nolimits_C \mc{M}_{n,\theta(n)}\\ \\
&= \sum_{k \in \Nb}\prod_{n \in C} \mc{M}_{n,k} \;+\; \sum_{[\theta] \in \left|\prod_C \Nb \right|_\nonstd}\prod\nolimits_C \mc{M}_{n,\theta(n)}.
\end{align*}
In the last expression, we use the notation $\prod_{n \in C} \mc{M}_{n,k}$ to emphasize that $k$ is fixed and that the cohesive product is taken over the sequence $\mc{M}_{0,k}, \mc{M}_{1,k}, \dots$.  Put
\begin{align*}
\mc{S}_k \;&=\; \prod_{n \in C} \mc{M}_{n, 6k} + \prod_{n \in C} \mc{M}_{n, 6k+1} + \prod_{n \in C} \mc{M}_{n, 6k+2} + \prod_{n \in C} \mc{M}_{n, 6k+3} + \prod_{n \in C} \mc{M}_{n, 6k+4} & \text{for each $k$,}\\ \\
\mc{R}_k \;&=\; \prod_{n \in C} \mc{M}_{n, 6k+5} & \text{for each $k$,}\\ \\
\mc{J} \;&=\; \sum_{[\theta] \in \left|\prod_C \Nb \right|_\nonstd}\prod\nolimits_C \mc{M}_{n,\theta(n)}
\end{align*}
to write $\prod_C \mc{L}_n$ in the desired form
\begin{align*}
\prod\nolimits_C \mc{L}_n \quad=\quad (\mc{S}_0 + \mc{R}_0 + \mc{S}_1 + \mc{R}_1 + \cdots) + \mc{J}.
\end{align*}

We now describe the $\mc{M}_{n,k}$.
The orders that are used to build up the $\mc{S}_k$ are straightforward to describe and directly reflect the desired structure of $\mc{S}_k$.
For each $n$ and $k$, set
\begin{gather*}
\mc{M}_{n, 6k} \;=\; \mc{M}_{n, 6k+2} \;=\; \mc{M}_{n, 6k+4} \;=\; k+2\\
\mc{M}_{n, 6k+1}  \;=\; \mc{M}_{n, 6k+3} \;=\; \Qb,
\end{gather*}
where $k+2$ is the usual presentation of the linear order $\bm{k+2}$ and $\Qb$ is the usual presentation of $(\Qb; <)$.  Note that the successor relation is uniformly computable for these orders.  Then
\begin{align*}
\prod_{n \in C}\mc{M}_{n, 6k} \;=\; \prod_{n \in C}\mc{M}_{n, 6k+2} \;=\; \prod_{n \in C}\mc{M}_{n, 6k+4} \;\iso\; \bm{k+2}\\
\end{align*}
for each $n$ and $k$ because $k+2$ is finite (see~\cite{CohPowJournal}*{Section~2})  and
\begin{align*}
\prod_{n \in C}\mc{M}_{n, 6k+1}  \;=\; \prod_{n \in C}\mc{M}_{n, 6k+3} \;\iso\; \eta
\end{align*}
for each $n$ and $k$ by~\cite{CohPowJournal}*{Proposition~3.4}.  Therefore we indeed have $\mc{S}_k \;\iso\; (\bm{k+2}) + \eta + (\bm{k+2}) + \eta + (\bm{k+2})$ for each $k$.

We now describe how to compute $\mc{M}_{n, 6k+5}$ for all $n$ and $k$.  Each $\mc{M}_{n, 6k+5}$ will either have order-type $\omega\cdot\bm{\ell}$ for some $\ell > 0$ or order-type $\omega\cdot\bm{\ell} + \bm{q}$ for some $\ell > 0$ and $q > k$.  Moreover, the successor relation on $\mc{M}_{n, 6k+5}$ will be c.e.\ uniformly in $n$ and $k$.

Think of $k$ as being fixed and of computing the sequence $(\mc{M}_{n, 6k+5} : n \in \Nb)$ of linear orders.  The goal is to ensure that $\mc{R}_k = \prod_{n \in C} \mc{M}_{n, 6k+5}$ has a maximum element if $A_k$ has an infinite column and that $\mc{R}_k$ does not have a maximum element if $B_k$ has an infinite column.  To do this, the columns of $A_k$ play finite sequences of length $k+1$ at the top of $\mc{M}_{n, 6k+5}$ for larger and larger $n$, and the columns of $B_k$ play $\omega$-sequences at the top of $\mc{M}_{n, 6k+5}$ for larger and larger $n$.  We impose priority on the plays, meaning that a lower priority column is not allowed to play in $\mc{M}_{n, 6k+5}$ if a higher priority column has already played there.  In detail, the computation proceeds as follows.

At stage $0$, $\mc{M}_{n, 6k+5}$ starts as an $\omega$-sequence for every $n$.  Declare $\la 0, x \ra \in |\mc{M}_{n, 6k+5}|$ for every $x$, with $\la 0, x \ra \prec_{\mc{M}_{n, 6k+5}} \la 0, y \ra$ if and only if $x < y$.  Furthermore, enumerate $\la\la 0, x \ra, \la 0, x+1 \ra\ra$ as a successor pair in $\mc{M}_{n, 6k+5}$ for every $x$.

At stage $s > 0$, let $i < s$ be least such that either $A_k(\la i, s \ra) = 1$ or $B_k(\la i, s \ra) = 1$.  If there is no such $i$, then declare $\la s, x \ra \notin |\mc{M}_{n, 6k+5}|$ for all $x$ and all $n$.  Go on to stage $s+1$.

Suppose there is such an $i$.  If $A_k(\la i, s \ra) = 1$, then we have noticed that column $i$ of $A_k$ has a new element $s$.  In this case, $A_k$ plays a copy of $\bm{k+1}$ at the top of $\mc{M}_{n, 6k+5}$ for the least $n > i$ such that neither $A_k$ nor $B_k$ has yet played in $\mc{M}_{n, 6k+5}$ on account of a column $j \leq i$.  Specifically, declare $\la s, x \ra \in |\mc{M}_{n, 6k+5}|$ for every $x \leq k$, with $\la s, x \ra \prec_{\mc{M}_{n, 6k+5}} \la s, y \ra$ if and only if $x < y$.  Declare $\la s, x \ra \notin |\mc{M}_{n, 6k+5}|$ for every $x > k$.  Declare $p \prec_{\mc{M}_{n, 6k+5}} \la s, x \ra$ for all $x \leq k$ and all $p$ already in $|\mc{M}_{n, 6k+5}|$ at the start of stage $s$.  Enumerate $\la\la s, x \ra, \la s, x+1 \ra\ra$ as a successor pair in $\mc{M}_{n, 6k+5}$ for every $x < k$.  Finally, if $\mc{M}_{n, 6k+5}$ had a maximum element $p$ at the start of stage $s$ (which is if $A_k$ was last to play in $\mc{M}_{n, 6k+5}$), then also enumerate $\la p, \la s, 0 \ra\ra$ as a successor pair in $\mc{M}_{n, 6k+5}$.  Declare $\la s, x \ra \notin |\mc{M}_{m, 6k+5}|$ for all $x$ and all $m \neq n$.  We say that \emph{$A_k$ has played in $\mc{M}_{n, 6k+5}$ on account of column $i$} or that \emph{column $i$ of $A_k$ played in $\mc{M}_{n, 6k+5}$}.  Go on to stage $s+1$.

If $A_k(\la i, s \ra) = 0$ but $B_k(\la i, s \ra) = 1$, then we have noticed that column $i$ of $B_k$ has a new element $s$.  In this case, $B_k$ plays a copy of $\omega$ at the top of $\mc{M}_{n, 6k+5}$ for the least $n > i$ such that neither $A_k$ nor $B_k$ has yet played in $\mc{M}_{n, 6k+5}$ on account of a column $j \leq i$.  Specifically, declare $\la s, x \ra \in |\mc{M}_{n, 6k+5}|$ for every $x$, with $\la s, x \ra \prec_{\mc{M}_{n, 6k+5}} \la s, y \ra$ if and only if $x < y$.  Declare also $p \prec_{\mc{M}_{n, 6k+5}} \la s, x \ra$ for all $x$ and all $p$ already in $|\mc{M}_{n, 6k+5}|$ at the start of stage $s$.  Enumerate $\la\la s, x \ra, \la s, x+1 \ra\ra$ as a successor pair in $\mc{M}_{n, 6k+5}$ for every $x$.  Finally, if $\mc{M}_{n, 6k+5}$ had a maximum element $p$ at the start of stage $s$ (which is if $A_k$ was last to play in $\mc{M}_{n, 6k+5}$), then also enumerate $\la p, \la s, 0 \ra\ra$ as a successor pair in $\mc{M}_{n, 6k+5}$.  Declare $\la s, x \ra \notin |\mc{M}_{m, 6k+5}|$ for all $x$ and all $m \neq n$.  We say that \emph{$B_k$ has played in $\mc{M}_{n, 6k+5}$ on account of column $i$} or that \emph{column $i$ of $B_k$ played in $\mc{M}_{n, 6k+5}$}.  Go on to stage $s+1$.

For a given $n$, $\mc{M}_{n, 6k+5}$ starts as an order of type $\omega$, and then only the columns $i < n$ of $A_k$ and $B_k$ can play further elements into $\mc{M}_{n, 6k+5}$.  Moreover, for each $i < n$, once either of $A_k$ or $B_k$ plays in $\mc{M}_{n, 6k+5}$ on account of column $i$, neither does so on account of column $i$ again.  Thus $\mc{M}_{n, 6k+5}$ is a finite sum of orders of type $\omega$ and orders of type $\bm{k+1}$, starting with an order of type $\omega$.  In particular, $\mc{M}_{n, 6k+5}$ either has order-type $\omega\cdot\bm{\ell}$ for some $\ell > 0$ or order-type $\omega\cdot\bm{\ell} + \bm{q}$ for some $\ell > 0$ and $q > k$.

We now show that $\prod_{n \in C} \mc{M}_{n, 6k+5}$ has a maximum element if $A_k$ has an infinite column and does not have a maximum element if $B_k$ has an infinite column.

Suppose first that $A_k$ has an infinite column and that $i$ is the least infinite column of $A_k$.  In this case, by disjointness, $B_k$ does not have an infinite column.  Thus the columns $j < i$ of $A_k$ are finite, and the columns $j \leq i$ of $B_k$ are finite.  So $A_k$ plays in $\mc{M}_{n, 6k+5}$ for finitely many $n$ on account of the columns $j < i$, and $B_k$ plays in $\mc{M}_{n, 6k+5}$ for finitely many $n$ on account of the columns $j \leq i$.  Let $n_0$ be such that for all $n > n_0$, the columns $j < i$ of $A_k$ and the columns $j \leq i$ of $B_k$ never play in $\mc{M}_{n, 6k+5}$.  Then for each $n > n_0$, column $i$ of $A_k$ eventually plays a copy of $\bm{k+1}$ at the top of $\mc{M}_{n, 6k+5}$, after which no further elements are played in $\mc{M}_{n, 6k+5}$.  To see this, suppose that column $i$ of $A_k$ has played in $\mc{M}_{m, 6k+5}$ for all $m$ with $n_0 < m < n$ by some stage $s_0$.  By increasing $s_0$ if necessary, assume it is greater than all the elements of $\{s : \exists j < i\; A_k(\la j, s \ra) = 1\} \cup \{s : \exists j \leq i\; B_k(\la j, s \ra) = 1\}$.  Let $s > s_0$ be least with $A_k(\la i, s \ra)=1$.  Then at stage $s$, $i$ is least such that $A_k(\la i, s \ra)=1$ or $B_k(\la i, s \ra)=1$.  Thus at stage $s$, column $i$ of $A_k$ plays a copy of $\bm{k+1}$ at the top of $\mc{M}_{n, 6k+5}$ if it has not done so already.  Once column $i$ of $A_k$ plays in $\mc{M}_{n, 6k+5}$ for an $n > n_0$, no further elements are ever added to $\mc{M}_{n, 6k+5}$.  No column $j < i$ of $A_k$ or column $j \leq i$ of $B_k$ ever plays in $\mc{M}_{n, 6k+5}$ by choice of $n_0$.  After column $i$ of $A_k$ plays in $\mc{M}_{n, 6k+5}$, no column $j \geq i$ of $A_k$ or of $B_k$ is allowed to play in $\mc{M}_{n, 6k+5}$ by the priority restriction.  Therefore the function
\begin{align*}
\varphi(n) &=
\begin{cases}
\text{the top element of the sequence of length $k+1$} & \text{if $n > n_0$}\\
\text{played by column $i$ of $A_k$ in $\mc{M}_{n, 6k+5}$}\\ \\
\ua & \text{if $n \leq n_0$}
\end{cases}
\end{align*}
is partial computable, and $\varphi(n)$ is the maximum element of $\mc{M}_{n, 6k+5}$ when $n > n_0$.  It follows that $[\varphi]$ is the maximum element of $\prod_{n \in C} \mc{M}_{n, 6k+5}$ by \cref{thm-LosProdParam}.

Suppose instead that $B_k$ has an infinite column and that $i$ is the least infinite column of $B_k$.  In this case, $A_k$ does not have an infinite column.  Arguing as in the previous case, there is an $n_0$ such that for each $n > n_0$, column $i$ of $B_k$ eventually plays a copy of $\omega$ at the top of $\mc{M}_{n, 6k+5}$, after which no further elements are played in $\mc{M}_{n, 6k+5}$.  Therefore the linear order $\mc{M}_{n, 6k+5}$ has no maximum element for all $n > n_0$.  It follows that $\prod_{n \in C} \mc{M}_{n, 6k+5}$ has no maximum element by \cref{thm-LosProdParam}.

We have now computed the linear orders $(\mc{M}_{n,k} : n,k \in \Nb)$.  Taking $\mc{L}_n = \sum_{k \in \Nb}\mc{M}_{n,k}$ for each $n$, we have that
\begin{align*}
\prod\nolimits_C \mc{L}_n \quad=\quad \sum_{k \in \Nb} (\mc{S}_k + \mc{R}_k) + \mc{J},
\end{align*}
Where $\mc{S}_k$ and $\mc{R}_k$ for each $k$ and $\mc{J}$ are defined as above.  We have shown that $\mc{S}_k \;\iso\; (\bm{k+2}) + \eta + (\bm{k+2}) + \eta + (\bm{k+2})$ for each $k$, that $\mc{R}_k = \prod_{n \in C} \mc{M}_{n, 6k+5}$ has a maximum element if $A_k$ has an infinite column, and that $\mc{R}_k$ has no maximum element if $B_k$ has an infinite column.  We need to show that, for every $k$, every non-maximum element of $\mc{R}_k$ has a successor, and that $\mc{J}$ does not have a finite block of size $\geq\! 2$.

Fix $k$ and consider a non-maximum element $[\varphi]$ of $\mc{R}_k$.  Then $\varphi(n)$ is not the maximum element of $\mc{M}_{n, 6k+5}$ for almost every $n \in C$ by \cref{thm-LosProdParam}.  The successor relation on the orders $\mc{M}_{n, 6k+5}$ is uniformly computable.  Therefore, we may partially compute the function
\begin{align*}
\psi(n) =
\begin{cases}
\text{the successor of $\varphi(n)$ in $\mc{M}_{n, 6k+5}$} & \text{if $\varphi(n)\da$ is not maximum in $\mc{M}_{n, 6k+5}$}\\
\ua & \text{otherwise}.
\end{cases}
\end{align*}
By assumption, $\psi(n)$ is defined for almost every $n \in C$, therefore $[\psi]$ is the successor of $[\varphi]$ in $\mc{R}_k$ by \cref{thm-LosProdParam}.

Now we show that $\mc{J}$ does not have a finite block of size $\geq\! 2$.  Consider an element $[\varphi]$ of a summand $\prod_C \mc{M}_{n,\theta(n)}$ of $\mc{J}$, where $[\theta] \in \left|\prod_C \Nb \right|_\nonstd$.  Recall that $\lim_{n \in C}\theta(n) = \infty$ in this case.  By cohesiveness, there is a $d < 6$ such that $\theta(n) \equiv d \mod 6$ for almost every $n \in C$.  If $d$ is $1$ or $3$, then $\mc{M}_{n,\theta(n)} = \Qb$ for almost every $n \in C$.  It follows from \cref{thm-LosProdParam} that $\prod_C \mc{M}_{n,\theta(n)}$ is a (countable) dense linear order without endpoints, so $\prod_C \mc{M}_{n,\theta(n)} \;\iso\; \eta$.  Thus $[\varphi]$ is a block of size $1$ in this case.

Say that an element $x$ of a linear order \emph{has $p$ successors} if there are $x_0, \dots, x_p$ with $x = x_0 \prec x_1 \prec \cdots \prec x_p$, where $x_{i+1}$ is the successor of $x_i$ for each $i < p$.  In this situation, say also that $x_i$ is the \emph{$i$\textsuperscript{th} successor of $x$}.  Define $x$ to \emph{have $p$ predecessors} and the \emph{$i$\textsuperscript{th} predecessor of $x$} in the analogous way.  In the remaining cases $d \in \{0,2,4,5\}$, we show that for every $p$, $[\varphi]$ has either $p$ successors or $p$ predecessors (or both) in $\prod_C \mc{M}_{n,\theta(n)}$.  Therefore $[\varphi]$ is in an infinite block.  Fix $p$.  If $d \in \{0,2,4\}$, then $\mc{M}_{n,\theta(n)}$ is a finite linear order of size $> 2p$ for all sufficiently large $n \in C$ because $\lim_{n \in C}\theta(n) = \infty$.  Similarly, if $d = 5$ then for all sufficiently large $n \in C$, either $\mc{M}_{n,\theta(n)} \iso \omega\cdot\bm{\ell}$ for some $\ell > 0$ or $\mc{M}_{n,\theta(n)} \iso \omega\cdot\bm{\ell} + \bm{q}$ for some $\ell > 0$ and $q > 2p$.  The point is that in all cases, if $n \in C$ is sufficiently large, then every element of $\mc{M}_{n,\theta(n)}$ has either $p$ successors or $p$ predecessors.  Furthermore, the successor relation on $\mc{M}_{n,\theta(n)}$ is uniformly partial computable in $n$.  (The `partial' here is on account of the partiality of $\theta$.)  Therefore ``$\varphi(n)$ has $p$ successors in $\mc{M}_{n,\theta(n)}$'' and ``$\varphi(n)$ has $p$ predecessors in $\mc{M}_{n,\theta(n)}$'' are $\Sigma_1$ predicates.  Thus by cohesiveness, it is either the case that $\varphi(n)$ has $p$ successors in $\mc{M}_{n,\theta(n)}$ for almost every $n \in C$ or that $\varphi(n)$ has $p$ predecessors in $\mc{M}_{n,\theta(n)}$ for almost every $n \in C$ (or both).  Consider the `successors' case.  The `predecessors' case is similar.  Again, the successor relation on $\mc{M}_{n,\theta(n)}$ is uniformly partial computable.  Therefore we may partially compute functions $\psi_1, \dots, \psi_p$, where
\begin{align*}
\psi_i(n) =
\begin{cases}
\text{the $i$\textsuperscript{th} successor of $\varphi(n)$ in $\mc{M}_{n,\theta(n)}$} & \text{if $\varphi(n)$ has $i$ successors in $\mc{M}_{n,\theta(n)}$}\\
\ua & \text{otherwise}
\end{cases}
\end{align*}
for each $1 \leq i \leq p$.  Then each $\psi_i(n)$ is defined for almost every $n \in C$, and
\begin{align*}
[\varphi] \;\prec_{\prod_C \mc{M}_{n,\theta(n)}}\; [\psi_1] \;\prec_{\prod_C \mc{M}_{n,\theta(n)}}\; \cdots \;\prec_{\prod_C \mc{M}_{n,\theta(n)}}\; [\psi_p]
\end{align*}
are $p$ successors of $[\varphi]$ in $\prod_C \mc{M}_{n,\theta(n)}$.  This completes the proof.
\end{proof}

We now explain how to turn the sequence of linear orders defined in the above theorem into a single linear order exhibiting a Tennenbaum-like phenomenon.  The single linear order is essentially the sum of the orders constructed above with additional separators.  We now analyze orders of this form.

\begin{Lemma}\label{lem-SumInterval}
Let $(\mc{L}_n : n \in \Nb)$ be a uniformly computable sequence of linear orders.  Let
\begin{align*}
\mc{S} \quad=\quad \sum_{n \in \Nb} (\{0\} + \mc{L}_n) \quad=\quad \{0\} + \mc{L}_0 + \{0\} + \mc{L}_1 + \cdots.
\end{align*}
Let $f, g \colon \Nb \to \Nb$ be the computable functions given by $f(n) = \la 2n, 0 \ra$ (i.e., the $0$ to the left of $\mc{L}_n$) and $g(n) = \la 2n+2, 0 \ra$ (i.e., the $0$ to the right of $\mc{L}_n$).  Let $C$ be a cohesive set.  Then the interval $([f], [g])$ in $\prod_C \mc{S}$ is isomorphic to $\prod_C \mc{L}_n$.
\end{Lemma}

\begin{proof}
The linear order $\mc{S}$ has domain $\{\la 2n, 0 \ra : n \in \Nb\} \cup \{\la 2n+1, \ell \ra : n \in \Nb \,\land\, \ell \in |\mc{L}_n|\}$ and is ordered in the obvious way, so it is easy to identify each $\mc{L}_n$ with its corresponding copy $\{2n+1\} \times |\mc{L}_n|$ in $\mc{S}$.

Let $[\varphi]_{\prod_C \mc{S}}$ be an element of $\prod_C \mc{S}$ with $[f] \prec_{\prod_C \mc{S}} [\varphi] \prec_{\prod_C \mc{S}} [g]$.  Then  $(\forae n \in C)(f(n) \prec_{\mc S} \varphi(n) \prec_{\mc S} g(n))$, so $(\forae n \in C)(\la 2n, 0 \ra \prec_{\mc S} \varphi(n) \prec_{\mc S} \la 2n+2, 0 \ra)$, so $(\forae n \in C)(\varphi(n) \in \{2n+1\} \times |\mc{L}_n|)$.  Let 
\begin{align*}
\psi(n) =
\begin{cases}
\pi_1(\varphi(n)) & \text{if $\pi_0(\varphi(n)) = 2n+1$}\\
\ua & \text{otherwise}.
\end{cases}
\end{align*}
Then $[\varphi]_{\prod_C \mc{S}}$ corresponds to the element $[\psi]_{\prod_C \mc{L}_n}$ of $\prod_C \mc{L}_n$.  Conversely, let $[\psi]_{\prod_C \mc{L}_n}$ be an element of $\prod_C \mc{L}_n$.  Thus $C \subseteq^* \dom(\psi)$ and $\forall n\, (\psi(n)\da \,\imp\, \psi(n) \in |\mc{L}_n|)$.  Let $\varphi(n) \simeq \la 2n+1, \psi(n) \ra$ so that $[\varphi]_{\prod_C \mc{S}}$ is an element of $\prod_C \mc{S}$.  Then $(\forae n \in C)(\la 2n, 0 \ra \prec_{\mc S} \varphi(n) \prec_{\mc S} \la 2n+2, 0 \ra)$, so in fact $[\varphi]_{\prod_C \mc{S}}$ is an element of the interval $([f], [g])$ in $\prod_C \mc{S}$ corresponding to $[\psi]_{\prod_C \mc{L}_n}$.  These correspondences are inverses of each other and preserve order, so the interval $([f], [g])$ in $\prod_C \mc{S}$ is isomorphic to $\prod_C \mc{L}_n$.
\end{proof}

In the statement of \cref{lem-SumInterval}, it would be more pleasing to write $\mc{S} = \sum_{n \in \Nb} (\bm{1} + \mc{L}_n)$.  The reason for the somewhat awkward formulation is to give names to the singleton breakpoints to describe the functions $f$ and $g$.
Regardless, this construction yields our desired result.

\begin{Theorem}\label{thm-LONonRec}
There is a computable linear order $\mc{L}$ such that for every cohesive set $C$ and every presentation $\mc{P}$ of $\prod_C \mc{L}$, the double-jump $\mc{P}''$ has $\pa$-degree relative to $0''$.  Thus $\prod_C \mc{L}$ has no computable presentation.
\end{Theorem}

\begin{proof}
Let $(\mc{L}_n : n \in \Nb)$ be the uniformly computable sequence of linear orders from \cref{thm-ProdNonRec}.  So for every cohesive set $C$, the double-jump $\mc{Q}''$ of every presentation $\mc{Q}$ of the cohesive product $\prod_C \mc{L}_n$ has $\pa$-degree relative to $0''$.  Let $\mc{L} = \sum_{n \in \Nb} (\{0\} + \mc{L}_n)$, and let $f, g \colon \Nb \to \Nb$ be the computable functions $f(n) = \la 2n, 0 \ra$ and $g(n) = \la 2n+2, 0 \ra$ as in the statement of \cref{lem-SumInterval}.  Let $C$ be a cohesive set, and let $\mc{P}$ be a presentation of $\prod_C \mc{L}$.  Let $a, b \in |\mc{P}|$ be the elements corresponding to $[f]$ and $[g]$ respectively.  Then the interval $(a,b)$ in $\mc{P}$ is isomorphic to $\prod_C \mc{L}_n$ by \cref{lem-SumInterval}.  Thus $\mc{P}$ computes a presentation $\mc{Q}$ of $\prod_C \mc{L}_n$, and therefore $\mc{P}'' \geqT \mc{Q}''$, which has $\pa$-degree relative to $0''$.  So $\mc{P}''$ has $\pa$-degree relative to $0''$.
\end{proof}

\begin{Corollary}
There is a computable structure $\mc{L}$ such that for every cohesive set $C$, $\prod_C \mc{L}$ has no degree. In other words, $\dgsp(\prod_C \mc{L})$ has no minimum element.
\end{Corollary}

\begin{proof}
The $\mc{L}$ described above has the property that, for every cohesive set $C$, $\mathbf{0} \notin \dgsp(\prod_C \mc{L})$.
It follows from~\cref{richter} that for every cohesive set $C$,  $\dgsp(\prod_C \mc{L})$ has no minimum element.
\end{proof}

\section{Further questions}\label{sec-q}

In addition to Questions~\ref{q-Pi3} and~\ref{q-omega} from the introduction, a number of further questions concerning the degree spectra of cohesive powers come to mind.  We relay a few of them here.

\begin{FQuestions*}\
\begin{enumerate}[(1)]
\setcounter{enumi}{4}
\item What other Turing degrees can be realized as degrees of cohesive powers in the sense of \cref{cor-degDJ}?

\medskip

\item What can be encoded into cohesive powers of linear orders?  Which degree spectra of linear orders are possible to achieve in this way?

\medskip

\item Is \cref{prop-Delta3Pres} best possible when $k>2$? Is there a structure $\mathcal{G}$ such that for every strictly $\Delta_k$ cohesive set $C$, $\prod_C\mc{G}$ has no $\Delta_k$ copies?
\end{enumerate}
\end{FQuestions*}

\section*{Acknowledgements}
We thank Leszek Ko{\l}odziejczyk and Patrick Lutz for helpful discussions.

\section*{Funding}
This project was partially supported by EPSRC grant EP/T031476/1.  Part of this research was performed while the second author was visiting the Simons Laufer Mathematical Sciences Institute (SLMath, formerly MSRI) during spring 2024, which was supported by the National Science Foundation (Grant No. DMS-1928930).

\bibliographystyle{amsplain}
\bibliography{TennenbaumForCohesivePowers}

\vfill

\end{document}